\documentclass[11pt]{article}

\usepackage{caption}
\DeclareCaptionLabelSeparator{twospace}{. }
\usepackage{indentfirst}
\usepackage{mathrsfs}
\usepackage{enumerate}
\usepackage[a4paper,margin=1in]{geometry}  
\usepackage[ruled,linesnumbered,vlined]{algorithm2e}
\usepackage{amsmath,amsthm,tikz,float}

\usepackage[colorlinks=true,citecolor=black,linkcolor=black,urlcolor=blue]{hyperref}
\usepackage{mathrsfs}
\usepackage{amssymb}

\usepackage{amsfonts}
\usepackage[shortlabels]{enumitem}
\setenumerate[1]{itemsep=0pt,partopsep=0pt,parsep=\parskip,topsep=0pt}
\usepackage{amsmath,amsthm,tikz,float,amssymb}
\usepackage[T1]{fontenc}
\usepackage[utf8]{inputenc}
\usepackage{authblk}

\usepackage{setspace}

\newtheorem{theorem}{Theorem}[section]
\newtheorem{lemma}[theorem]{Lemma}
\newtheorem{definition}[theorem]{Definition}

\newtheorem{proposition}[theorem]{Proposition}
\newtheorem{conjecture}[theorem]{Conjecture}

\newtheorem{claim}{Claim}

\newenvironment{MainThmproof}
{\begin{proof}[{Proof of Theorem \ref{MainResult}}]}
{\end{proof}}

\newenvironment{MainLemproof}
{\begin{proof}[{Proof of Lemma \ref{lemMain}}]}
{\end{proof}}

\begin{document}


\title{The Tur\'{a}n number of the triangular pyramid of 4-layers}
\author{Hangdi Chen$^{1,}$\thanks{chenhangdi188@126.com} ,  Yaojun Chen$^{2,}$\thanks{yaojunc@nju.edu.cn} ,  Xiutao Zhu$^{3,}$\thanks{zhuxiutao@nuaa.edu.cn}\\
    {\small $^1$Fujian Key Laboratory of Financial Information Processing, Putian University, }\\ 
    {\small Putian, 351100, P.R. China}\\
    {\small $^2$School of Mathematics, Nanjing University,}\\ {\small Nanjing, 210093, P.R. China}\\
	{\small $^3$School of Mathematics, Nanjing University of Aeronautics and Astronautics,}\\ {\small Nanjing, 211106, P.R. China}
}
\date{}
\maketitle

\begin{abstract}
The Tur\'{a}n number $\textup{ex}(n,H)$ of a graph $H$ is the maximum number of edges in any $H$-free graph on $n$ vertices. The triangular pyramid of $k$-layers, denoted by $TP_k$, is a generalization of a triangle. The Tur\'an problems of a triangular pyramid with small layers have been studied widely by  Liu (E-JC, 2013), Xiao, Katona, Xiao and Zamora (DAM, 2022), Ghosh, Gy\H{o}ri, Paulos, Xiao and Zamora (DAM, 2022). Moreover, Ghosh et al.  conjectured  that $\textup{ex}(n, TP_4)=\frac{1}{4}n^2+\Theta(n^{\frac{4}{3}})$. In this note, we confirm this conjecture. 
\end{abstract}



\section{Introduction}

In this note, all graphs considered are simple. For a graph $G$, let $V(G)$ be the \emph{vertex set} of $G$ and  $E(G)$ be the \emph{edge set} of $G$.  The number of edges of $G$ is denoted by $e(G)=|E(G)|$. The \emph{chromatic number} of $G$ is denoted by $\chi (G)$. Let $K_n$, $P_n$ and $C_n$ denote the \emph{complete graph}, \emph{path} and \emph{cycle} on $n$ vertices, respectively. Denote by $K_{a,b}$ the \emph{complete bipartite graph} whose part sizes are $a$ and $b$. A graph $G$ is called $H$-free if it does not contain a copy of $H$ as a subgraph.  

The \emph{Tur\'{a}n number} of a graph $H$, denoted by $\textup{ex}(n,H)$, is the maximum number of edges that an  $n$-vertex $H$-free graph can have. Denote by $T_k(n)$ the $n$-vertex balanced complete $k$-partite graph (the sizes of any  two partites differ by at most $1$). The first result about the Tur\'an problem can be traced back to Mantel's theorem, who determined the Tur\'an number of a triangle. 
\begin{theorem}
(\textup{Mantel \cite{Mantel}}).
For all $n$, $\textup{ex}(n,K_3)=\lfloor\frac{n^2}{4}\rfloor$ and $K_{\lfloor\frac{n}{2}\rfloor,\lceil\frac{n}{2}\rceil}$ is the unique extremal graph.  
\end{theorem}

Later, Tur\'an extended Mantel's result by determining the Tur\'an number $\textup{ex}(n,K_{r+1})$ of any clique $K_{r+1}$ \cite{Turan1941}. Since then, the study of the Tur\'an problem has attracted a lot of attention and been a very active field. And for any non-bipartite graph $H$, the asymptotic value of $\textup{ex}(n,H)$  was determined by Erd\H{o}s, Stone and Simonovits.
\begin{theorem}\label{ErdStoneSimoThm}
(\textup{Erd\H{o}s et al. \cite{Erd1966Simonovits,Erd1946Stone}}). Let $\chi (H)=r+1\ge 3$. Then
\[\textup{ex}(n,H)=e(T_r(n))+o(n^2).\]
\end{theorem}
Although Theorem \ref{ErdStoneSimoThm} already gives the asymptotic value of $\textup{ex}(n,H)$, it is still very difficult to determine its exact value.  In this note, we will focus on a class of non-bipartite graph, which is called the Triangular Pyramid graph $TP_k$, defined below. 

\begin{definition}
Define the Triangular Pyramid $TP_k$ with $k$ layers as follows: Firstly, draw $k+1$ paths in layers such that the $i$-th layer is an $i$-vertex path $P_i=y_1^iy_2^i\cdots y_i^i$, where $i\in\{1,2,\ldots,k+1\}$.  Secondly, for any two consecutive layers, say the $i$-th and $(i+1)$-th, add edges $y_t^iy_t^{i+1}$ and $y_t^iy_{t+1}^{i+1}$, $t\in\{1,2,\ldots,i\}$ (see Figure \ref{TPexam}). 
\end{definition}

\begin{figure}[H]
\begin{minipage}[t]{0.5\textwidth}
    \centering
    \begin{tikzpicture}
        \coordinate (a1) at (6,6.5);
        \coordinate (a2) at (5,5.5);
		\coordinate (a3) at (7,5.5);
		\coordinate (a4) at (4,4.5); 
		\coordinate (a5) at (6,4.5);
		\coordinate (a6) at (8,4.5);

        \foreach \i in {1,2,3,4,5,6}
        {
            
            \fill (a\i) circle[radius=3pt];
        }
    
        \draw (a2) -- (a1) -- (a3) -- (a6) -- (a5)-- (a4)-- (a2);
\draw (a2) --  (a3)  -- (a5)-- (a2);

                  \node at (6,6.9) {$y_1^1$};
                  \node at (4.6,5.5) {$y_1^2$};
                  \node at (7.4,5.5) {$y_2^2$};
 \node at (3.6,4.5) {$y_1^3$};
 \node at (6,4.1) {$y_2^3$};
 \node at (8.4,4.5) {$y_3^3$};
\node at (6,3.5) {$TP_2$};	       
    \end{tikzpicture}
\end{minipage}
\begin{minipage}[t]{0.4\textwidth}
    \centering
\begin{tikzpicture}
        \coordinate (a1) at (6,6.9);
        \coordinate (a2) at (5.3,6.2);
		\coordinate (a3) at (6.7,6.2);
		\coordinate (a4) at (4.6,5.5); 
		\coordinate (a5) at (6,5.5);
		\coordinate (a6) at (7.4,5.5);
\coordinate (a7) at (3.9,4.8); 
		\coordinate (a8) at (5.3,4.8);
		\coordinate (a9) at (6.7,4.8);
\coordinate (a10) at (8.1,4.8); 
\coordinate (a11) at (3.2,4.1);
\coordinate (a12) at (4.6,4.1); 
		\coordinate (a13) at (6,4.1);
		\coordinate (a14) at (7.4,4.1);
\coordinate (a15) at (8.8,4.1);

        \foreach \i in {1,2,3,4,5,6,7,8,9,10,11,12,13,14,15}
        {
            
            \fill (a\i) circle[radius=3pt];
        }
    
        \draw (a2) -- (a1) -- (a3) -- (a6) -- (a5)-- (a4)-- (a2);
\draw (a2) --  (a3)  -- (a5)-- (a2);
 \draw (a6) -- (a10) -- (a15) -- (a14) -- (a13)-- (a12)-- (a11) -- (a7)-- (a4);
\draw (a7) --  (a8)  -- (a9)-- (a10);
  \draw (a4) -- (a8) -- (a13) -- (a9) -- (a6);
\draw (a7) -- (a12) -- (a8) -- (a5) -- (a9)-- (a14) -- (a10);
	 \node at (6,3) {$TP_4$};	       
    \end{tikzpicture}
\end{minipage}
  \caption{Triangular Pyramids with $2$ and $4$ layers respectively.}\label{TPexam}
\end{figure}
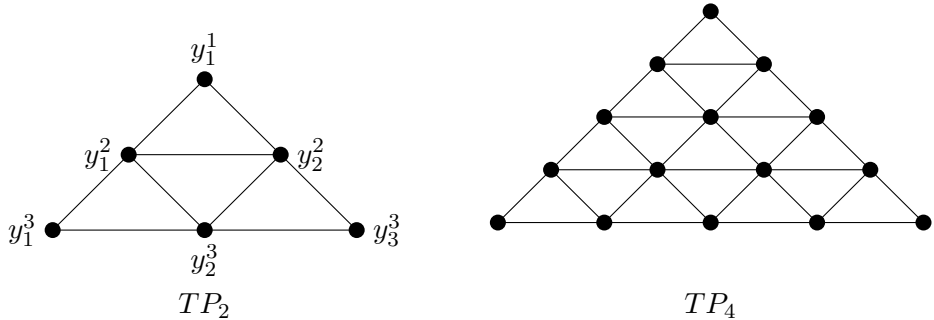

From this perspective, the triangle $K_3$ is a special pyramid $TP_1$ with $1$ layer. For the pyramid $TP_2$, Liu \cite{Liu2013Hong} determined the precise value of $\textup{ex}(n, TP_2)$ and characterized all extremal graphs when $n$ is large. Later, Xiao, Katona, Xiao and Zamora \cite{Xiao2022Zamora} determined the value of $\textup{ex}(n, TP_2)$ for all $n$.
For the triangular pyramid $TP_3$, Ghosh,  Gy\H{o}ri, Paulos, Xiao and Zamora \cite{Ghosh2022Xiao} proved the following asymptotic value of $\textup{ex}(n, TP_3)$ up to the linear factor. 

\begin{theorem}\label{Ghosh2022TP3}
(\textup{Ghosh et al. \cite{Ghosh2022Xiao}}).  The Tur\'{a}n number of $TP_3$ is 
\[\textup{ex}(n,TP_3)=\frac{n^2}{4}+n+o(n).\]
\end{theorem}

Moreover, they also \cite{Ghosh2022Xiao} posed a conjecture for $\textup{ex}(n, TP_4)$.

\begin{conjecture}\label{Conj2022Ghosh}
(\textup{Ghosh et al. \cite{Ghosh2022Xiao}}).  For $n$ sufficiently large, $\textup{ex}(n,TP_4)=\frac{1}{4}n^2+\Theta(n^{\frac{4}{3}})$.
\end{conjecture}

 In this note, we confirm Conjecture \ref{Conj2022Ghosh} and obtain the following result. 

\begin{theorem}\label{MainResult}
For $n$ sufficently large, the Tur\'{a}n number of $TP_4$ is 
\[\textup{ex}(n,TP_4)=\frac{1}{4}n^2+\Theta(n^{\frac{4}{3}}).\]
\end{theorem}
\section{Preliminaries}
In this section, we introduce some notations and results which will be used in the proof of Theorem \ref{MainResult}. Let $G$ be a simple graph, and $N_G(v)$ denote the neighborhood of $v$ consisting of all vertices adjacent to $v$. The degree of a vertex $v$, denoted by $d_G(v)$, is the size of $N_G(v)$, and the minimum degree of $G$ is denoted by $\delta (G)$.  For two disjoint subsets  $A$ and $B$ of $V(G)$, let $G[A]$ and $G[A,B]$ denote the subgraph induced by $A$ and the bipartite subgraph which consists of all edges with one endpoint in $A$ and the other endpoint in $B$, respectively. The number of edges of $G[A,B]$ is denoted by $e_G(A,B)$. When no confusion can occur, we will omit the subscript $G$. 


Let $F$ be a graph with $V(F)=\{x_i\mid 1\le i\le 10\}$ and $E(F)=\{x_ix_{i+1}\mid 1\le i\le 9\}\cup\{x_3x_8\}$, as shown in Figure \ref{Fpic}. 

\begin{figure}[H]

    \centering
    \begin{tikzpicture}
        \coordinate (a1) at (6,1.5);
        \coordinate (a2) at (7.3,2.3);
		\coordinate (a3) at (8.6,1.5);
		\coordinate (a4) at (8.6,0); 
		\coordinate (a5) at (7.3,-0.8);
		\coordinate (a6) at (6,0);
\coordinate (b1) at (4,1.5);
        \coordinate (b2) at (5,1.5);
\coordinate (c1) at (4,0);
        \coordinate (c2) at (5,0);
		
        \foreach \i in {1,2,3,4,5,6}
        {
            
            \fill (a\i) circle[radius=3pt];
        }
       \foreach \i in {1,2}
        {
            
            \fill (b\i) circle[radius=3pt];
            \fill (c\i) circle[radius=3pt];
        }
        \draw (a1) -- (a2) -- (a3) -- (a4) -- (a5)-- (a6)-- (a1);
       \draw (b1) -- (b2) -- (a1);
       \draw (c1) -- (c2) -- (a6);
		
                  \node at (4,1.8) {$x_1$};
                  \node at (5,1.8) {$x_2$};
                  \node at (6,1.8) {$x_3$};
 \node at (7.3,2.6) {$x_4$};
 \node at (8.9,1.6) {$x_5$};
 \node at (8.9,-0.1) {$x_6$};
 \node at (7.3,-1.1) {$x_7$};
\node at (6,-0.3) {$x_8$};
\node at (5,-0.3) {$x_9$};
\node at (4,-0.3) {$x_{10}$};
    \end{tikzpicture}
 \caption{The graph $F$.}\label{Fpic}
\end{figure}
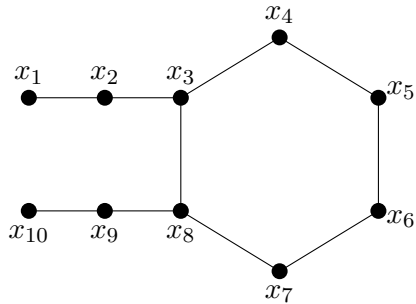

We need the following upper bound for the Tur\'{a}n number of $F$. 

\begin{proposition}\label{propMain}
$\textup{ex}(n,F)\le \textup{ex}(n,C_6)+10n$.
\end{proposition}
\begin{proof}
Let $G$ be an $F$-free graph on $n$ vertices. Suppose that $e(G)> \textup{ex}(n,C_6)+10n$. We delete the vertices of degree less than $10$ in $G$ recursively and get the subgraph $G'$. Note that $\delta(G')\ge 10$. It is easy to get that $e(G')> \textup{ex}(n,C_6)$ and hence there is a copy of $C_6$ in $G'$. Since $\delta(G')\ge 10$, we can extend this copy of $C_6$ to $F$, a contradiction. 
\end{proof}

\begin{proposition}\label{propTP4}
Suppose that $s\ge 10$ and $t\ge 10$ are two integers. If $G$ is a graph obtained from a $K_{s,t}$ by adding an $F$ to one part of $K_{s,t}$, then $G$ contains a $TP_4$.
\end{proposition}
\begin{proof}
Suppose that $G$ satisfies the above condition and let $H=TP_4$. Note that $\chi (H)=\chi(TP_4)=3$. Consider a proper 3-coloring of the graph $H$, and denote by $V_i$  the set of vertices assigned color $i$ for each $i\in\{1,2,3\}$. It is easy to check that $H[V_1\cup V_2]= F$, and $V_3$ is an independent set of $H$. By the definition of the graph $G$, $G$ contains a $TP_4$.
\end{proof}

 The following two classical results are due to Erd\H{o}s and Simonovits. 

\begin{theorem}\label{erd1983SimoThm}
(\textup{Erd\H{o}s and Simonovits \cite{Erd1983Simonovits}}).There exist positive constants $c = c(r)$,$\gamma= \gamma(r)$ such that any graph $G$ on n vertices with $e=e(G)\ge cn^{2-\frac{1}{r}}$ edges contains at least  $\gamma\frac{e^{r^2}}{n^{2r^2-2r}}$  copies  of $K_{r,r}$. 
\end{theorem}

\begin{theorem}\label{erdSimoThm}
(\textup{Simonovits \cite{1966Simonovits}}). Let $k\ge 2$ and suppose that $H$ is a graph with $\chi(H)=k+1$. If $G$ is an $H$-free graph with $e(G)\ge e(T_k(n))-o(n^2)$, then $G$ can be obtained from $T_k(n)$ by adding and deleting $o(n^2)$ edges. 
\end{theorem}

\section{Proof of Theorem \ref{MainResult}}

Let $h(n)=e(T_2(n))$. First, we use a standard trick to reduce the problem to $TP_4$-free graphs with large minimum degree (Lemma \ref{lemMain}). Before proving Lemma \ref{lemMain}, we will use it to prove Theorem \ref{MainResult} first, and then present the proof of Lemma \ref{lemMain}.

\begin{lemma}\label{lemMain}
Let $n$ be sufficiently large and $G$ be an $n$-vertex TP$_4$-free graph with $\delta(G)\ge h(n)-h(n-1)$. Then \[e(G)\le h(n)+\textup{O}(n^{\frac{4}{3}}).\]
\end{lemma}

\begin{MainThmproof}
In \cite{Ghosh2022Xiao}, the authors have showed a construction and proved 
$$\textup{ex}(n,TP_4)\ge\frac{1}{4}n^2+\textup{O}(n^{\frac{4}{3}}),$$ 
so  we only need to prove the upper bound:
\[ex(n,TP_4)\le h(n)+\textup{O}(n^{\frac{4}{3}}).\]
 
 Let $G$ be a $TP_4$-free graph on $n$ vertices, where $n$ is large enough. If $e(G)\le h(n)$ or $\delta(G)\ge h(n)-h(n-1)$, then there is nothing to prove by Lemma \ref{lemMain}. So we may assume $e(G)>h(n)$ and there exists a vertex $u_n\in V(G)$ with $d_G(u_n)\le h(n)-h(n-1)-1$.  We define a process  as follows: Let $G_n=G$ and $G_{n-1}=G_n-\{u_n\}$. Suppose
 $G_i$ is already defined. If there exists a vertex $u_i\in G_i$ with $d_{G_i}(u_i)\le h(i)-h(i-1)-1$, then let $G_{i-1}=G_i-\{u_i\}$.  Repeat this iteration until we get $G_t$ such that $\delta(G_t)\ge h(t)-h(t-1)$.
 
During the process, we have 
 \[e(G_{i-1})=e(G_{i})-d_{G_i}(u_i)\ge e(G_i)-h(i)+h(i-1)+1.\]
  Since $e(G_n)=e(G)\ge h(n)$, by above inequality, we have 
\begin{align*}
e(G_t)&\ge h(n)-\sum_{k=t+1}^{n} (h(k)-h(k-1)-1)\\
&\ge h(t)+n-t.
\end{align*}
Suppose that $n\ge n_0+\binom{n_0}{2}$, here $n_0$ satisfies the condition in Lemma \ref{lemMain}. If $t<n_0$, then $e(G_t)\ge n-t>\binom{n_0}{2}$. On the other hand, $e(G_t)\le \binom{t}{2}<\binom{n_0}{2}$, a contradiction. So we have $t\ge n_0$. Because of $\delta(G_t)\ge h(t)-h(t-1)$ and Lemma \ref{lemMain}, we have $e(G_t)\le h(t)+\textup{O}(t^{\frac{4}{3}})$. Hence, 
\begin{align*}
e(G)&=e(G_t)+\sum_{k=t+1}^{n} d_{H_k}(u_k)\\
&\le  h(t)+\textup{O}(t^{\frac{4}{3}})+h(n)-h(t)-(n-t)\\
&\le h(n)+\textup{O}(n^{\frac{4}{3}}).
\end{align*}
The proof of Theorem \ref{MainResult} is complete. 
\end{MainThmproof}

Now we give the proof of Lemma \ref{lemMain}. 
\begin{MainLemproof}
Note that $\delta(G)\ge h(n)-h(n-1)\ge\lfloor\frac{n}{2}\rfloor$.  Assume that $X_1$ and $X_2$ is a partition of $V(G)$ with $e(X_1,X_2)$ maximal. If $e(G)<h(n)$, we complete the proof. So suppose that $e(G)\ge h(n)$. Let $\varepsilon>0$ be sufficiently small. Because of $\chi (TP_4)=3$, by Theorem \ref{erdSimoThm}, we have that 
\begin{equation}\label{eqn3}
e(G[X_1])+e(G[X_2])< \varepsilon n^2.
\end{equation}
For the same reason, we also have that 
\begin{equation}\label{eqn4}
e(G[X_1,X_2])> h(n)-\varepsilon n^2.
\end{equation}
For all $i\in\{1,2\}$,  we will bound $|X_i|$ by presenting the following claim. 
\begin{claim}\label{clm3}
For every $i\in\{1,2\}$, $|X_i|\le (\frac{1}{2}+\sqrt{\varepsilon})n$.
\end{claim}
\begin{proof}
Without loss of generality, suppose that $|X_1|=(\frac{1}{2}+r)n$ and $|X_2|=(\frac{1}{2}-r)n$, where $r\ge 0$. By the inequality (\ref{eqn3}), we have
\[e(G)\le e(G[X_1])+e(G[X_2])+e(G[X_1,X_2])\le \varepsilon n^2+|X_1||X_2|\le \varepsilon n^2+\left(\frac{1}{4}-r^2\right)n^2.\]
Since $e(G)\ge h(n)$, it follows that $r\le \sqrt{\varepsilon}$.  Hence, $|X_i|\le (\frac{1}{2}+\sqrt{\varepsilon})n$, where $i\in\{1,2\}$.
\end{proof}

Let $\lambda=\sqrt[3]{4\varepsilon}$, where $\varepsilon$ is small enough such that  $\lambda\ll 1$.  Let $R_i=\{u\in X_i~|~|N_G(u)\cap X_i|>\lambda n\}$ for all $i\in\{1,2\}$. Set $R=R_1\cup R_2$ and let $W_i=X_i\setminus R$ for all $i\in\{1,2\}$. By the inequality (\ref{eqn3}) and $\lambda\ge 2\sqrt{\varepsilon}$, we get
\[|R|<\frac{2\varepsilon n^2}{\lambda n}\le \frac{\lambda}{2} n.\]
In fact, we will show $|R|$ is a constant. 
\begin{claim}\label{clm2}
$|R|=\textup{O}(1)$. 
\end{claim}
\begin{proof}
For every $u\in X_i$ and $i\in\{1,2\}$, we have $|N(u)\cap X_{3-i}|\ge |N(u)\cap X_i|$. Otherwise, $X_i\setminus\{u\}$ and $X_{3-i}\cup\{u\}$ is a new partition of $V(G)$ such that \[e(X_i\setminus\{u\},X_{3-i}\cup \{u\})>e(X_i,X_{3-i}),\] a contradiction.  For every fixed $u\in R$, since $W_i=X_i\setminus R$ and $|R|< \frac{\lambda}{2} n$, we have $|N(u)\cap W_i|>\frac{\lambda}{2} n$. Let $n_0=\lceil\frac{\lambda}{2} n\rceil$ and $A_i\subseteq N(u)\cap W_i$ be a set of size $n_0$. By the inequality (\ref{eqn4}) and $\lambda=\sqrt[3]{4\varepsilon}$,  we have \[e(G[A_1,A_2])>|A_1||A_2|-\varepsilon n^2>n_0^2-\frac{\lambda^3}{4}n^2\ge (1-\lambda)n_0^2.\]
By Theorem \ref{erd1983SimoThm}, we deduce that the graph $G[A_1\cup A_2]$  contains at least $C_\varepsilon(2n_0)^{30}=C'_\varepsilon n^{30}$ copies of $K_{15,15}$, where $C_\varepsilon$ and $C'_\varepsilon$ are constants depending on $\varepsilon$.

Let $\mathcal{T}$ be the set  of all copies of $K_{15,15}$ in $G$. Note that $|\mathcal{T}|\le n^{30}$. Define a bipartite graph $\mathcal{F}$ with the bipartition $(\mathcal{T},R)$ and $V(\mathcal{F})=\mathcal{T}\cup R$. For any $S\in \mathcal{T}$ and $u\in R$, $Su\in E(\mathcal{F})$ if and only if $V(S)\subseteq N_G(u)$. By the above proof, we have $d_{\mathcal{F}}(u)\ge C'_\varepsilon n^{30}$ for all $u\in R$. If there exists an $S\in \mathcal{T}$ such that $d_{\mathcal{F}}(S)\ge 5$, then $G$ contains a $TP_4$, a contradiction. So $d_{\mathcal{F}}(S)\le 4$ for all $S\in \mathcal{T}$. Hence, %
\[C'_\varepsilon n^{30}|R|\le \sum_{u\in R}d_{\mathcal{F}}(u)\le e(\mathcal{F})\le \sum_{S\in \mathcal{T}} d_{\mathcal{F}}(S)\le 4|\mathcal{T}|\le 4n^{30}.\] 
So $|R|\le \frac{4}{C'_\varepsilon}$, we are done.
\end{proof}

\begin{claim}\label{clm4}
For every $u\in W_i$ and every $i\in\{1,2\}$, $|N(u)\cap W_{3-i}|> \lfloor\frac{n}{2}\rfloor-\frac{3}{2}\lambda n$.
\end{claim}
\begin{proof}
By symmetry, suppose that $i=1$ and $u\in W_1$. Note that $|N(u)\cap X_1|\le \lambda n$ and  $\delta(G)\ge \lfloor\frac{n}{2}\rfloor$. By the definition of $R$, we have
\[ \left\lfloor\frac{n}{2}\right\rfloor\le d_G(u)=|N(u)\cap X_2|+|N(u)\cap X_1|\le |N(u)\cap X_2|+\lambda n.\]
So $ |N(u)\cap X_2|\ge \lfloor\frac{n}{2}\rfloor-\lambda n$.  Since $|R_2|\le |R|$ and $|R|< \frac{\lambda n}{2}$, we have 
\begin{align*}
|N(u)\cap W_2| &=|N(u)\cap X_2|-|N(u)\cap R_2|\\
&\ge  |N(u)\cap X_2|-|R|\\
&>  \left\lfloor\frac{n}{2}\right\rfloor-\frac{3}{2}\lambda n.
\end{align*}
This proves Claim \ref{clm4}.
\end{proof}

\begin{claim}\label{clm5}
For every $i\in\{1,2\}$, $G[W_i]$ is $F$-free.  
\end{claim}
\begin{proof}
Suppose not,  we assume that $G[W_1]$ contains a copy of  $F$ as a subgraph. Let $Y=V(F)$. By Claim \ref{clm3}, we have 
\begin{equation}\label{eqnMain}
|W_i|\le |X_i|<\frac{n}{2}+\sqrt{\varepsilon} n.
\end{equation}
 Note that the number of common neighbors of $Y$ in $W_2$ is at least
\begin{align*}
&~~~~\frac{1}{|Y|}\left(\sum_{v\in Y}|N(v)\cap W_2|-(|Y|-1)|W_2|\right)\\
&>\frac{1}{|Y|}\left(|Y|\left(\left\lfloor\frac{n}{2}\right\rfloor-\frac{3}{2}\lambda n\right)-(|Y|-1)\left(\frac{n}{2}+\sqrt{\varepsilon} n\right)\right)\\
&\ge  \frac{1}{|Y|}\left(|Y|\left(\frac{n-1}{2}-\frac{3}{2}\lambda n\right)-|Y|\left(\frac{1}{2}+\sqrt{\varepsilon}\right)n+\left(\frac{1}{2}+\sqrt{\varepsilon}\right)n\right)\\
&\ge -\left(\sqrt{\varepsilon}+\frac{3}{2}\lambda\right)n-\frac{1}{2}+\left(\frac{1}{2}+\sqrt{\varepsilon}\right)\frac{n}{|Y|}\\
&\ge \left(\frac{1}{20}-\frac{3}{2}\lambda-\frac{9}{10}\sqrt{\varepsilon}\right)n-\frac{1}{2}.
\end{align*}
The first inequality follows from the inequality (\ref{eqnMain}) and Claim \ref{clm4}. Since  $\lambda$ and $\varepsilon$ is sufficiently small, there exists a set $C\subset W_2$ such that $|C|=10$ and $G[Y,C]$ is a complete bipartite graph.  By Proposition \ref{propTP4}, $G[Y,C]$ contains a copy of $TP_4$, a contradiction.
\end{proof}
Now we estimate the upper bound of $e(G)$. Note that 
\[e(G[W_1,W_2])\le e(T_2(n))=h(n).\]
By Claim \ref{clm5} and Proposition \ref{propMain}, we have
\[e(G[W_1])+e(G[W_2])\le 2\textup{ex}(n,C_6)+10n\le\textup{O}(n^{\frac{4}{3}}).\]
The last inequality follows from a result on the upper bound of $\textup{ex}(n,C_{2k})$ due to Bondy and Simonovits \cite{Bondy1974Simonovits}.
By Claim \ref{clm2}, the number of edges that contain some vertex in $R$ is at most 
\[|R|n\le\textup{O}(n).\]
Combining the above inequality, we deduce that 
\[e(G)\le e(G[W_1,W_2])+e(G[W_1])+e(G[W_2])+|R|n \le h(n)+\textup{O}(n^{\frac{4}{3}}).\]
The proof of Lemma \ref{lemMain} is completed.
\end{MainLemproof}

\section*{Declarations}
The authors declare that they have no known competing financial interests or personal relationships that could have appeared to influence the work reported in this paper.
\section*{\bf\Large Availability of Data and Materials}  
Not applicable.
\section*{Acknowledgments}
This research is supported by National Key R\&D Program of China under grant number 2024YFA1013900, NSFC under grant numbers 12471327 and 12401454, Natural Science Foundation of Fujian Province under grant number 2024J01875, Science-Technology Foundation of Putian University under grant number 2023059, Basic Research Programma of Jiangsu Province under grant number BK20241361, Jiangsu Funding Program for Excellent Postdoctoral Talent under grant number 2024ZB179 and State‐sponsored Postdoctoral Researcher program under number GZB20240976.



\begin{thebibliography}{10}

\bibitem{Bondy1974Simonovits}
J. A. Bondy and M. Simonovits,
{\em Cycles or even length in graphs},
J. Combin. Theory Ser. B {\bf16} (1974), 97-105.


\bibitem{Erd1966Simonovits}
P. Erd\H{o}s and M. Simonovits,
{\em A limit theorem in graph theory},
Studia Sci. Math. Hungar. {\bf1} (1966), 51-57.

\bibitem{Erd1983Simonovits}
P. Erd\H{o}s and M. Simonovits,
{\em Supersaturated graphs and hypergraphs},
Combinatorica  {\bf3} (1983), 181-192.

\bibitem{Erd1946Stone}
P. Erd\H{o}s and A. Stone,
{\em On the structure of linear graphs},
Bull. Ame. Math. Soc.  {\bf52} (1946), 1087-1091.




\bibitem{Ghosh2022Xiao}
D. Ghosh, E. Gy\H{o}ri, A. Paulos, C. Xiao and O. Zamora,
{\em The Tur\'{a}n number of the triangular pyramid of $3$-layers},
Discrete Appl. Math. {\bf317} (2022), 75-85.

\bibitem{Liu2013Hong}
H. Liu,
{\em Extremal graphs for blow-ups of cycles and trees},
Electron. J. Combin.  {\bf20} (2013), 1-16.




\bibitem{Mantel}
W. Mantel, {\em Problem 28}, soln. by H. Gouventak, W. Mantel, J. Teixeira de Mattes, F. Schuhand and W.A. Wythoff. Wiskundige Opgaven {\bf10} (1907), 60-61.


\bibitem{1966Simonovits}
M. Simonovits,
{\em A method for solving extremal problems in graph theory, stability problems, in: Theory of Graphs},
Proc. Colloq. Tihany  (1966), 279-319.

\bibitem{Turan1941}
  P. Tur\'{a}n,
{\em On an extremal problem in graph theory (in Hungrarian)},
Mat.  Fiz. Lapok. {\bf48} (1941), 436-452.

\bibitem{Xiao2022Zamora}
C. Xiao, G. Katona, J. Xiao and O. Zamora,
{\em The Tur\'{a}n number of the square of a path},
Discrete Appl. Math. {\bf307} (2022), 1-14.





\end{thebibliography}
\end{document}